\documentclass[a4,12pt,leqno]{amsart}
\usepackage{amsmath}
\usepackage{amsfonts}
\usepackage{amssymb}
\usepackage{mathrsfs}

\newcommand{\R}{\mathbb{R}}

\swapnumbers
\theoremstyle{plain}
\newtheorem{thm}[equation]{Theorem}
\newtheorem{lemma}[equation]{Lemma}

\theoremstyle{definition}
\newtheorem{defn}[equation]{Definition}

\newtheorem{exmp}[equation]{Example}

\theoremstyle{remark}

\numberwithin{equation}{section}

\title{Ahlfors-David regular sets and bilipschitz maps}
\author{Pertti Mattila}
\address{Pertti Mattila: Department of Mathematics and Statistics\\FI-00014 University of Helsinki\\ Finland\\ } 
\email{pertti.mattila@helsinki.fi}
\date{}
\author{Pirjo Saaranen}
\address{Pirjo Saaranen:  Haaga-Helia\\ Hietakummuntie 1 A\\FI-00700 Helsinki\\Finland\\}
\email{pirjo.saaranen@haaga-helia.fi}
\thanks{Part of these results were included in the licentiate thesis of Pirjo Saaranen at the University 
of Jyv\"askyl\"a in 1991.}

\begin{document} 

\begin{abstract}
Given two Ahlfors-David regular sets in metric spaces, we study the question whether one of 
them has a subset bilipschitz equivalent with the other.
\end{abstract}

\maketitle

\section{Introduction}
In this paper we shall study Ahlfors-David regular subsets of metric spaces. 
Throughout $(X,d)$ and $(Y,d)$ will be metric spaces. For $E,F\subset X$ and $x\in X$ 
we shall denote by $d(E)$ the diameter of $E$, by $d(E,F)$ the distance between $E$ and $F$, and by $d(x,E)$ the 
distance from $x$ to $E$. The closed ball with center $x$ and radius $r$ 
is denoted by $B(x,r)$.

\begin{defn}\label{regular} Let $E\subset X$ and $0<s<\infty$. 
We say that $E$ 
is $s$-regular if it is closed and if there exists a Borel (outer) measure $\mu$ on $X$ 
and a constant $C_E, 1\leq C_E<\infty$, such that $\mu(X\setminus E)=0$ and
$$r^s\leq\mu(B(x,r))\leq C_Er^s\ \text{for all}\ x\in E, 0<r\leq d(E), r<\infty.$$
\end{defn}

Observe that this implies that the right hand inequality holds for all $x\in E, r>0$, and  
$$\mu(B(x,r))\leq 2^sC_Er^s\ \text{for all}\ x\in X, r>0.$$

We would get an equivalent definition (up to the value of $C_E$), if we would use the restriction of 
the $s$-dimensional Hausdorff measure on $E$, with $r^s$ on the left hand side replaced by $r^s/C_E$. 
When we shall speak about a regular set $E$, $\mu$ will always stand for a measure as above.

We remark that closed and bounded subsets of regular sets are compact, see Corollary 5.2 in [DS2]. Self 
similar subsets of $\R^n$ satisfying the open set condition are standard examples of regular sets, see 
[H].

A map $f:X\to Y$ is said to be bilipschitz if it is onto and there is a positive number $L$, called a 
bilipschitz constant of $f$,  such that 
$$d(x,y)/L\leq d(f(x),f(y))\leq Ld(x,y)\ \text{for all}\ x,y\in X.$$
The smallest such $L$ is denoted by bilip$(f)$.
Evidently any bilipschitz image of an $s$-regular set is $s$-regular. But 
two regular sets of the same dimension $s$ need not be bilipschitz equivalent. This is so even for very simple 
Cantor sets in $\R$, see [FM], [RRX] and [RRY] for results on the bilipschitz equivalence of such Cantor sets, 
and for [DS2] for extensive analysis of bilipschitz invariance properties of fractal type sets. 

The main content of this paper is devoted to the following question: suppose $E$ is $s$-regular and $F$ is 
$t$-regular. If $s<t$, does $F$ have a subset which is bilipschitz equivalent to $E$? 
In this generality the answer
is obviously no due to topological reasons; $E$ could be connected and $F$ 
totally disconnected. We shall prove in Theorem \ref{bilip1} that the answer is yes for any $0<s<t$ if 
$E$ is a standard 
$s$-dimensional Cantor set in some $\R^n$ with $s<n$. We shall also prove in Theorem \ref{bilip2} 
that the answer is always yes if $s<1$. In Section 4 we show that if $E$ and $F$ as above are subsets 
of $\R^n$ and $s$ is sufficiently small, then a bilipschitz map $f$ with $f(E)\subset F$ can be defined in 
the whole of $\R^n$. We don't know if this holds always when $0<s<1$.

In the last section of the paper we shall discuss sub- and supersets of regular sets. It follows from 
Theorem \ref{bilip1} that an $s$-regular set contains a $t$-regular subset for any $0<t<s$. In the other 
direction we shall show 
that if $E\subset X$ is $s$-regular and $X$ is $u$-regular, then for any $s<t<u$ there is a $t$-regular set 
$F$ such that $E\subset F\subset X$. 
On the other hand there are rather nice sets which do not contain any regular subsets: we shall construct 
a compact subset of $\R$ with positive Lebesgue measure which does not contain any $s$-regular subset for 
any $s>0$.

Regular sets in connection of various topics of analysis are discussed for example in [DS1] and [JW].

\section{Some lemmas on regular sets}

In this section we shall prove some simple lemmas on regular sets.  

\begin{lemma} \label{balls}
Let $0<s<\infty$ and let $E\subset X$ be $s$-regular. For every $0<r<R\leq d(E), R<\infty$, and $p\in E$ there exist 
disjoint closed balls $B(x_i,r),i=1,\dots,m$, such that $x_i\in E\cap B(p,R)$,
$$(5^sC_E)^{-1}(R/r)^s\leq m\leq 2^sC_E(R/r)^s$$
and
$$E\cap B(p,R)\subset\bigcup_{i=1}^mB(x_i,5r).$$
\end{lemma}

\begin{proof} By a standard covering theorem, see, e.g., Theorem 2.1 in [M], we can find disjoint balls 
$B(x_i,r), i=1,2,\dots$, such that $x_i\in E\cap B(p,R)$ and the balls $B(x_i,5r)$ cover $E\cap B(p,R)$. 
There are only finitely many,
say $m$, of these balls, since the disjoint sets $B(x_i,r)$ have all $\mu$ measure at least 
$r^s$, they are contained in $B(p,2R)$ which has measure at most $C_E(2R)^s$. More precisely, we have 
$$mr^s\leq\sum_{i=1}^m\mu(B(x_i,r))\leq\mu(B(p,2R))\leq C_E(2R)^s,$$
whence $m\leq 2^sC_E(R/r)^s$, and 

$$mC_E5^sr^s\geq\sum_{i=1}^m\mu(B(x_i,5r))\geq\mu(B(p,R))\geq R^s,$$
whence $m\geq (5^sC_E)^{-1}(R/r)^s$.
\end{proof}

For less than one-dimensional sets we can get more information:
 
\begin{lemma} \label{rings}
Let $0<s<1, C\geq1, R>0$, let $E\subset X$ be closed and bounded and let $\mu$ be a 
Borel measure on $X$ such that $\mu(X\setminus E)=0$ and that
$$\mu(B(x,r))\leq Cr^s\ \text{for all}\ x\in E, r>0,$$
and
$$\mu(B(x,r))\geq r^s\ \text{for all}\ x\in E, 0<r<R.$$
Let $D=(3C2^s)^{1/(1-s)}+1$. For every $0<r<R/(2D)$ there exist 
disjoint closed balls $B(x_i,r),i=1,\dots,m$, and positive numbers $\rho_i, r\leq\rho_i\leq Dr$, 
such that $m\leq Cd(E)^s/r^s,x_i\in E,x_j\not\in B(x_i,\rho_i)$ for $i<j$,
$$E\subset\bigcup_{i=1}^mB(x_i,\rho_i)\ \text{and}\ E\cap B(x_i,\rho_i+r)\setminus B(x_i,\rho_i) 
=\emptyset.$$
\end{lemma}

\begin{proof} Let $x_1\in E$. Denote
$$A_0=B(x_1,r), A_i=B(x_1,(i+1)r)\setminus B(x_1,ir), i=1,2,\dots.$$
If $E\cap A_1=\emptyset$, denote $\rho_1=r$. Otherwise,    
let $l$ be the largest positive integer such that $2lr<R$ and $E\cap A_i\not=\emptyset$ for $i=1,\dots,l$, 
say $y_i\in E\cap A_i$. Then 
$B(y_i,r)\subset A_{i-1}\cup A_i\cup A_{i+1}\subset B(x_1,2lr)$. Therefore
\begin{align*}
&lr^s\leq\sum_{i=1}^l\mu(B(y_i,r))\leq\sum_{i=1}^l\mu(A_{i-1}\cup A_i\cup A_{i+1})\leq\\
&3\mu(B(x_1,2lr))\leq3C2^sl^sr^s,
\end{align*}
whence $l^{1-s}\leq3C2^s$ and, since $s<1$, $l\leq(3C2^s)^{1/(1-s)}=D-1$. 
As $2(l+1)\leq 2D<R/r$, 
we conclude that $E\cap A_{l+1}=\emptyset$ by the maximality of $l$. Let $\rho_1=(l+1)r$. 
Then $r\leq\rho_1\leq Dr$ and $E\cap B(x_1,\rho_1+r)\setminus B(x_1,\rho_1)=\emptyset$. 
Let $x_2\in E\setminus B(x_1,\rho_1)=E\setminus B(x_1,\rho_1+r)$. Then the balls $B(x_1,r)$ 
and $B(x_2,r)$ are disjoint. Repeating the same argument as above with $x_1$ replaced by $x_2$ 
we find $\rho_2$ such that  $r\leq\rho_2\leq Dr$ and 
$E\cap B(x_2,\rho_2+r)\setminus B(x_2,\rho_2)=\emptyset$. After $k-1$ steps we choose
$$x_k\in E\setminus\bigcup_{i=1}^{k-1}B(x_i,\rho_i),$$
if this set is non-empty. As in the proof of Lemma \ref{balls} this process ends after some 
$m$ steps when $E$ is covered by the balls $B(x_i,\rho_i), i=1,\dots,m$. Also, as before, 
$m$ satisfies the required estimate $m\leq Cd(E)^s/r^s$.
\end{proof}

The following lemma will be needed to get bilipschitz maps in the whole $\R^n$.

\begin{lemma} \label{rings2}
Let $C\geq 1$ and $\lambda\geq9$. There are positive numbers $s_0=s_0(C,\lambda), 0<s_0<1$, and $D=D(C,\lambda)>1$, depending only 
on $C$ and $\lambda$, with the following property.

Let $0<s<s_0$, let $E\subset X$ be closed and bounded, let $R>0$ and let $\mu$ be a 
Borel measure on $X$ such that $\mu(X\setminus E)=0$ and that
$$\mu(B(x,r))\leq Cr^s\ \text{for all}\ x\in E, r>0,$$
and
$$\mu(B(x,r))\geq r^s\ \text{for all}\ x\in E, 0<r<R.$$
For every $0<r<R/D$ there exist 
disjoint closed balls $B(x_i,\lambda\rho_i/3),i=1,\dots,m$, such that $x_i\in E$, $r\leq \rho_i\leq Dr$,
$m\leq Cd(E)^s/r^s,$
$$E\subset\bigcup_{i=1}^mB(x_i,\rho_i)\ \text{and}\ E\cap B(x_i,\lambda \rho_i)\setminus B(x_i,\rho_i)
=\emptyset.$$
\end{lemma}
\begin{proof} The function $s\mapsto(1-3C\lambda^{2s}(\lambda^s-1))^{-1/s}$
is positive and increasing in some interval $(0,s_1)$, so it is bounded in some interval $(0,s_0)$. 
We choose $s_0$ and $D$ so that 
$$\lambda(1-3C\lambda^{2s}(\lambda^s-1))^{-1/s}\leq D\ \text{for}\ 0<s<s_0.$$
Set $c=\log D/\log\lambda$.

Let $x\in E$ and denote 
$$A_0=B(x,r), A_i=B(x,\lambda^ir)\setminus B(x,\lambda^{i-1}r), i=1,2,\dots.$$
If $E\cap A_1=\emptyset$, denote $r(x)=r$. Otherwise, let $l$ be the largest positive integer such that $l\leq c$ 
and that $E\cap A_i\not=\emptyset$ for $1=1,\dots,l$.
Then for $i=1,\dots,l$ there is $y_i\in E\cap A_i$ with 
$B(y_i,\lambda^{i-2}r)\subset A_{i-1}\cup A_i\cup A_{i+1}$. By the choice of $c$, $\lambda^{l-2}r<Dr<R$. 
Hence 
\begin{align*}
&r^s\lambda^{-s}\frac{\lambda^{sl}-1}{\lambda^s-1}=r^s\sum_{i=1}^l\lambda^{s(i-2)}\leq\\
&\sum_{i=1}^l\mu(B(y_i,\lambda^{i-2}r))\leq\sum_{i=1}^l\mu(A_{i-1}\cup A_i\cup A_{i+1})\leq\\
&3\mu(E\cap B(x,\lambda^{l+1}r))\leq3C\lambda^{s(l+1)}r^s.
\end{align*}
This gives
$$(1-3C\lambda^{2s}(\lambda^s-1))\lambda^{sl}\leq1,$$
whence
$$\lambda^{l+1}\leq\lambda(1-3C\lambda^{2s}(\lambda^s-1))^{-1/s}\leq D.$$
Thus $l+1\leq c$ and we conclude that $E\cap A_{l+1}=\emptyset$. 
Let $r(x)=\lambda^lr$. We have now shown that for any $x\in E$ there is $r(x), r\leq r(x)\leq Dr$, such that 
$E\cap B(x,\lambda r(x))\setminus B(x,r(x))=\emptyset.$

Let $M_1=\sup\{r(x):x\in E\}$. Choose $x_1\in E$ with $r(x)>M_1/2$, and then inductively 
$$x_{j+1}\in E\setminus\bigcup_{i=1}^jB(x_i,r(x_i))\ \text{with}\ r(x_{j+1})>M_1/2$$
as long as possible. Thus we get points $x_i\in E$ and radii $r(x_i), r\leq r(x_i)\leq Dr,$ 
for $i=1,\dots,k_1$ such that $r(x_i)/2\leq r(x_j)\leq 2r(x_i)$, $x_j\not\in B(x_i,r(x_i))$ for $i<j$, and 
$$\{x\in E: r(x)>M_1/2\}\subset\bigcup_{i=1}^{k_1}B(x_i,r(x_i)).$$
If for some $l=1,2,\dots$ the points $x_1,\dots,x_{k_l}$ have been selected and
$E\setminus\bigcup_{i=1}^{k_l}B(x_i,r(x_i))\not=\emptyset$, let 
$$M_{l+1}=\sup\{r(x):x\in E\setminus\bigcup_{i=1}^{k_l}B(x_i,r(x_i))\},$$
choose $x_{k_{l+1}}\in E\setminus\bigcup_{i=1}^{k_l}B(x_i,r(x_i))$ with $r(x_{k_{l+1}})>M_{l+1}/2$, and so on. 
This process will end for some $l=p$. Thus we get points $x_1,\dots,x_m\in E, m=k_p$,  
such that, with $\rho_i=r(x_i), r\leq \rho_i\leq Dr$, for $i<j, x_j\not\in B(x_i,\rho_i)$ and 
$r_j\leq 2\rho_i$,  
$$E\subset\bigcup_{i=1}^mB(x_i,\rho_i)\ \text{and}\ E\cap B(x_i,\lambda \rho_i)\setminus B(x_i,\rho_i)=\emptyset.$$
To show that the balls $B(x_i,\lambda\rho_i/3)$ are disjoint, let $i<j$. 
Then $\rho_j\leq 2\rho_i$ and $x_j\in E\cap(\R^n\setminus B(x_i,\rho_i))=E\cap(\R^n\setminus B(x_i,\lambda \rho_i))$. 
So $d(x_i,x_j)>\lambda \rho_i$ and $(\lambda/3)(\rho_i+\rho_j)\leq\lambda \rho_i<d(x_i,x_j)$, which implies that 
$B(x_i,\lambda\rho_i/3)\cap B(x_j,\lambda\rho_j/3)=\emptyset$. The required estimate $m\leq Cd(E)^s/r^s$ follows as before.
\end{proof}

\section{Bilipschitz maps}

In this section we begin to prove the bilipschitz equivalences mentioned in the introduction. It is easy to get explicit 
bounds for the bilipschitz constants of the maps from the proofs. In Theorem \ref{bilip1} bilip$(f)$ is bounded by a 
constant 
depending only on $s,t,n$ and $C_E$. In Theorems \ref{bilip2} and \ref{bilip3}, if $C_E, C_F$ and  $d(E)/d(F)$ 
(interpreted as 
0 if $F$ is unbounded) are all $\leq C$, then bilip$(f)\leq L$ where $L$ depends only on $s,t$ and $C$, and 
also on $n$ in Theorem 
\ref{bilip3}. If $E$ and $F$ are bounded, this dependence on the diameters is seen by first observing that we may assume 
that $d(F)\leq d(E)$; otherwise $F$ can be replaced in the 
proofs by $F\cap B(p,d(E)/2)$ for any $p\in F$. Secondly, changing the metrics to $d_E(x,y)=d(x,y)/d(E)$ and 
$d_F(x,y)=d(x,y)/d(F)$, we have $d(E)=d(F)=1$, the regularity constants don't change and a bilipschitz constant $L$ 
changes to $Ld(E)/d(F)$. 

For any $0<t<n$ 
we shall define some standard $t$-dimensional Cantor sets in $\R^n$. Define $0<d<1/2$ by $2^nd^t=1$. Let 
$Q\subset\R^n$ be a closed cube of side-length $a$. Let $Q_1,\dots,Q_{2^n}\subset Q$ be the closed cubes of side-length 
$da$ in the corners of $Q$. Continue this process. Then $C(t,a)$ is defined as 
$$C(t,a)=\bigcap_{k=1}^{\infty}\bigcup_{i_1\dots i_k}Q_{i_1\dots i_k},$$
where $i_j=1,\dots 2^n$ and each $Q_{i_1\dots i_k}$ is a closed cube of sidelength $d^ka$ such that 
$Q_{i_1\dots i_ki}, i=1,\dots,2^n$, are contained in the corners of $Q_{i_1\dots i_k}$. It is well known and easy to 
prove that $C(t,a)$ is $t$-regular, it is also a particular case of a self similar set satisfying the open set as 
considered in [H].

\begin{thm}\label{bilip1}
Let $E\subset X$ be a bounded $s$-regular set and $0<t<s$. Then there is a $t$-regular subset $F$ of $E$ and 
a bilipschitz map $f:F\to C(t,d(E))$ where $C(t,d(E))$ is a Cantor subset of $\R^n$ with $t<n$ as above. 
Moreover, $C_F\leq C$ where $C$ depends only $s,t,n$ and $C_E$.
\end{thm}
\begin{proof}
We may assume that $d(E)=1$. Choose a sufficiently large integer $N$ so that denoting 
$d=2^{-Nn/t}$, i.e., $2^{Nn}d^t=1$, we have $d<1/3$ and $d^{s-t}<(15^sC_E)^{-1}=:c$. Then we can write $C(t,1)$ as
$$C(t,1)=\bigcap_{k=1}^{\infty}\bigcup_{i_1\dots i_k}Q_{i_1\dots i_k}$$
where each $Q_{i_1\dots i_k}, 1\leq i_j\leq2^{Nn}$, is a closed cube of side-length $d^k$ such that
$Q_{i_1\dots i_ki_{k+1}}\subset Q_{i_1\dots i_k}$. By Lemma \ref{balls} we can find disjoint 
balls $B(x_i,3d), x_i\in E,i=1,\dots,m$, such that $m\geq cd^{-s}>d^{-t}=2^{Nn}$. Now we keep the 
first $2^{Nn}$ points $x_i$ and forget about the others. Repeating this argument with $E$ replaced by $E\cap B(x_i,d)$
and so on, we can choose points 
$$x_{i_1\dots i_ki_{k+1}}\in E\cap B(x_{i_1\dots i_k},d^k), 1\leq i_j\leq2^{Nn},$$
such that the balls $B(x_{i_1\dots i_ki},3d^{k+1}), i=1,\dots,2^{Nn}$, are disjoint subsets of $B(x_{i_1\dots i_k},3d^k)$.
Then for $1\leq l<k$,
\begin{equation}\label{1}
d(x_{i_1\dots i_l},x_{i_1\dots i_k})\leq\sum_{j=l}^{k-1}d(x_{i_1\dots i_j},x_{i_1\dots i_{j+1}})
\leq\sum_{j=l}^{k-1}d^j<2d^l
\end{equation}
as $d<1/2$. Denote
$$F=\bigcap_{k=1}^{\infty}\bigcup_{i_1\dots i_k}B(x_{i_1\dots i_k},3d^k).$$
Then $F\subset E$. Let $y_{i_1\dots i_k}$ be the center of $Q_{i_1\dots i_k}$ and denote
$$F_k=\{x_{i_1\dots i_k}:i_j=1,\dots,2^{Nn}, j=1,\dots,k\}$$
and
$$C_k=\{y_{i_1\dots i_k}:i_j=1,\dots,2^{Nn}, j=1,\dots,k\}.$$
Define the maps 
$$f_k:F_k\to C_k\ \text{by}\ f(x_{i_1\dots i_k})=y_{i_1\dots i_k}.$$

We check now that $f_k$ is bilipschitz with a constant depending only on $s,t,n$ and $C_E$. Let 
$x=x_{i_1\dots i_k}, x'=x_{j_1\dots j_k}\in F_k$ with $x\not=x'$. Let $l\geq1$ be such that $i_1=j_1,\dots,i_l=j_l$ and 
$i_{l+1}\not=j_{l+1}$; if $i_1\not=j_1$ the argument is similar. 
Then by (\ref{1}) $x\in B(x_{i_1\dots i_li_{l+1}},2d^{l+1})\cap B(x_{i_1\dots i_l},2d^l)$ and 
$x'\in B(x_{j_1\dots j_lj_{l+1}},2d^{l+1})\cap B(x_{i_1\dots i_l},2d^l)$. Since the balls  
$B(x_{i_1\dots i_li_{l+1}},3d^{l+1})$ and $B(x_{j_1\dots j_lj_{l+1}},3d^{l+1})$ are disjoint, we get
that $d^{l+1}\leq d(x,x')\leq4d^l$. Letting $y=y_{i_1\dots i_k}$ and $y'=y_{j_1\dots j_k}$ we see from the 
construction of $C(t,1)$ that $(1-2d)d^l\leq |y-y'|\leq \sqrt{n}d^l$. Hence
$$d(f_k(x),f_k(x'))=|y-y'|\leq (\sqrt{n}/d)d(x,x')$$
and 
$$d(f_k(x),f_k(x'))=|y-y'|\geq((1-2d)/4)d(x,x').$$
Denote $L=\max\{\sqrt{n}/d,4/(1-2d)\}$.

If $x\in F$ there is a unique sequence $(i_1,i_2,\dots)$ such that $x\in B(x_{i_1\dots i_k},3d^k)$ for all 
$k=1,2,\dots$. Let $y\in C(t,1)$ be the point for which $y\in Q_{i_1\dots i_k}$ for all 
$k=1,2,\dots$. Then $y=\lim_{k\to\infty}y_{i_1\dots i_k}=\lim_{k\to\infty}f_k(x_{i_1\dots i_k})$. We define the map 
$f:F\to C(t,1)$ by setting $f(x)=y$. If also $x'=\lim_{k\to\infty}x_{j_1\dots j_k}$ and 
$y'=\lim_{k\to\infty}y_{j_1\dots j_k}$ we have 
\begin{align*}
&d(f(x),f(x'))=\lim_{k\to\infty}(f_k(x_{i_1\dots i_k}),f_k(x_{j_1\dots j_k}))\\
&\leq\lim_{k\to\infty}Ld(x_{i_1\dots i_k},x_{j_1\dots j_k})=Ld(x,x')
\end{align*}
and similarly $d(f(x),f(x'))\geq d(x,x')/L$. Obviously, $f(F)=C(t,1)$.
The last statement, $C_F\leq C$, of the theorem follows immediately from the fact that $L$ depends only $s,t,n$ and $C_E$.
\end{proof}

Next we turn to study less than one-dimensional sets. 

\begin{thm}\label{bilip2}
Let $E\subset X$ be $s$-regular and $F\subset Y$  $t$-regular with $0<s<1$ and $s<t$. 
Suppose that either $E$ is bounded or both $E$ and $F$ are unbounded. 
Then there is a bilipschitz map $f:E\to f(E)\subset F$. 

\end{thm}
\begin{proof}
We shall first consider the case where both $E$ and $F$ are bounded. By the remarks in the beginning of this section, 
we then may assume that $d(E)=d(F)=1$.  Let $D=(3C_E2^s)^{1/(1-s)}+1$.
Choose $d$ so small that 
$$0<d^{t-s}<(2^s15^tD^sC_EC_F)^{-1}\ \text{and}\ 2Dd<1.$$ 
We shall show that there exist $s$-regular sets 
$E_{i_1\dots i_k}$, points $x_{i_1\dots i_k}\in E, y_{i_1\dots i_k}\in F$ and radii 
$\rho_{i_1\dots i_k}$ where
$$1\leq i_j\leq m_{i_0\dots i_{j-1}}, j=1,\dots,k,\ \text{with}\ m_{i_0\dots i_{j-1}}\leq C_E2^sD^s/d^s, i_0=0,$$
such that for all $k=1,2,\dots,$
\begin{align*}
&E=\bigcup_{i_1\dots i_k}E_{i_1\dots i_k},\\
&E_{i_1\dots i_ki_{k+1}}\subset E_{i_1\dots i_k},\\
&d^k\leq\rho_{i_1\dots i_k}\leq Dd^k,\\
&x_{i_1\dots i_k}\in E_{i_1\dots i_k}\subset B(x_{i_1\dots i_k},\rho_{i_1\dots i_k}),\\
&E\cap B(x_{i_1\dots i_k},\rho_{i_1\dots i_k}+d^k)\setminus B(x_{i_1\dots i_k},\rho_{i_1\dots i_k})=\emptyset,\\
&d(E_{i_1\dots i_k},E_{j_1\dots j_k})\geq d^k\ \text{if}\ i_k\not=j_k,\\
&y_{i_1\dots i_ki_{k+1}}\in F\cap B(y_{i_1\dots i_k},d^k),\\
&B(y_{i_1\dots i_ki_{k+1}},2d^{k+1})\subset B(y_{i_1\dots i_k},2d^k),\\
&B(y_{i_1\dots i_k},3d^k)\cap B(y_{j_1\dots j_k},3d^k)=\emptyset\ \text{if}\ i_k\not=j_k.
\end{align*}

By Lemma \ref{rings} we find $x_i\in E$ and $\rho_i, d\leq \rho_i\leq Dd$, with $i=1,\dots,m_0, m_0\leq C_E/d^s$,
such that the balls $B(x_i,d)$ are disjoint, $x_j\not\in B(x_i,\rho_i)$ for $i<j$, 
$$E\subset\bigcup_{i=1}^{m_0}B(x_i,\rho_i)$$
and
$$E\cap B(x_i,\rho_i+d)\setminus B(x_i,\rho_i)=\emptyset.$$
By Lemma \ref{balls} we find  $y_i\in F$ with 
$i=1,\dots,n_0, n_0\geq (15^tC_Fd^t)^{-1}\geq C_E/d^s\geq m_0$
such that the balls $B(y_i,3d)$ are disjoint. We define
$$E_1=E\cap B(x_1,\rho_1)\ \text{and}\ E_i=E\cap B(x_i,\rho_i)\setminus\bigcup_{j=1}^{i-1}E_j\ \text{for}\ i\geq2.$$
Then the required properties for $k=1$ are readily checked.

Suppose then that for some $k\geq1, E_{i_1\dots i_k}, x_{i_1\dots i_k}\in E, 
y_{i_1\dots i_k}\in F$ and $\rho_{i_1\dots i_k}$ have been 
found with the asserted properties. Fix $i_1\dots i_k$.  We shall apply 
Lemma \ref{rings} with $E=E_{i_1\dots i_k}, R=d^k$, $r=d^{k+1}$ and $C=C_E$, recall that $2Dd<1$. Since 
$d(E_{i_1\dots i_k},E\setminus E_{i_1\dots i_k})\geq d^k$, we have 
$E\cap B(x,r)=E_{i_1\dots i_k}\cap B(x,r)$ for $x\in E_{i_1\dots i_k}$ and $0<r<d^k$, so this is possible. 
Thus we obtain $x_{i_1\dots i_ki}\in E_{i_1\dots i_k}$ and $\rho_{i_1\dots i_ki}, i=1,\dots,m_{i_0\dots i_k}$, such that
$m_{i_0\dots i_k}\leq C_Ed(E_{i_1\dots i_k})^s/d^{(k+1)s}\leq C_E2^sD^s/d^s$, the balls $B(x_{i_1\dots i_ki},d^{k+1})$ are 
disjoint, $x_{i_1\dots i_kj}\not\in B(x_{i_1\dots i_ki},\rho_{i_1\dotsi_ki})$ for $i<j$,
$$d^{k+1}\leq\rho_{i_1\dots i_ki}\leq Dd^{k+1},$$
$$E_{i_1\dots i_k}\subset \bigcup_{i=1}^{m_{i_0\dots i_k}}B(x_{i_1\dots i_ki},\rho_{i_1\dots i_ki})$$
and
$$E\cap B(x_{i_1\dots i_ki},\rho_{i_1\dots i_ki}+d^{k+1})\setminus B(x_{i_1\dots i_ki},\rho_{i_1\dots i_ki})=\emptyset.$$
Define
$$E_{i_1\dots i_k1}=E_{i_1\dots i_k}\cap B(x_{i_1\dots i_k1},\rho_{i_1\dots i_k1})$$
and
$$E_{i_1\dots i_ki}=E_{i_1\dots i_k}\cap B(x_{i_1\dots i_ki},\rho_{i_1\dots i_ki})
\setminus\bigcup_{j=1}^{i-1}E_{i_1\dots i_kj}\ 
\text{for}\ i\geq2.$$
Applying Lemma \ref{balls} we find points $y_{i_1\dots i_ki}\in F\cap B(y_{i_1\dots i_k},d^k), i=1,\dots,n_{i_0\dots i_k},$ 
with $n_{i_0\dots i_k}\geq(15^tC_Fd^t)^{-1}\geq C_E2^sD^s/d^s\geq m_{i_0\dots i_k}$ such that the balls 
$B(y_{i_1\dots i_ki},3d^{k+1}),i=1,\dots,n_{i_0\dots i_k}$, are disjoint. 
Then the required properties are easily checked.

Set
$$A_k=\{x_{i_1\dots i_k}:i_j=1,\dots,m_{i_0\dots i_{j-1}}, j=1,\dots,k\}$$
and
$$B_k=\{y_{i_1\dots i_k}:i_j=1,\dots,m_{i_0\dots i_{j-1}}, j=1,\dots,k\}.$$
Define the maps
$$f_k:A_k\to B_k\ \text{by}\ f(x_{i_1\dots i_k})=y_{i_1\dots i_k}.$$
We check now that $f_k$ is bilipschitz with a constant depending only on $s,t,C_E$ and $C_F$. Let 
$x=x_{i_1\dots i_k}, x'=x_{j_1\dots j_k}\in A_k$ with $x\not=x'$. Let $l\geq1$ be such that $i_1=j_1,\dots,i_l=j_l$ and 
$i_{l+1}\not=j_{l+1}$; if $i_1\not=j_1$ the argument is similar. Then, as in (3.2) in the proof of Theorem \ref{bilip1},
$x\in E_{i_1\dots i_li_{l+1}}\cap B(x_{i_1\dots i_l},2Dd^l)$ and 
$x'\in E_{j_1\dots j_lj_{l+1}}\cap B(x_{i_1\dots i_l},2Dd^l)$. Since 
$d(E_{i_1\dots i_li_{l+1}},E_{j_1\dots j_lj_{l+1}})\geq d^{l+1}$, we get
that $d^{l+1}\leq d(x,x')\leq4Dd^l$. Letting $y=y_{i_1\dots i_k}$ and $y'=y_{j_1\dots j_k}$, we have 
$y\in B(y_{i_1\dots i_li_{l+1}},2d^{l+1})\cap B(y_{i_1\dots i_l},2d^l)$ and 
$y'\in B(y_{j_1\dots j_lj_{l+1}},2d^{l+1})\cap B(y_{i_1\dots i_l},2d^l)$.  
Hence, as $B(y_{i_1\dots i_li_{l+1}},3d^{l+1})\cap B(y_{j_1\dots j_lj_{l+1}},3d^{l+1})=\emptyset$, 
$d^{l+1}\leq d(y,y')\leq4d^l$, 
$$d(f_k(x),f_k(x'))=d(y,y')\leq (4/d)d(x,x')$$
and 
$$d(f_k(x),f_k(x'))=d(y,y')\geq(d/(4D))d(x,x').$$
Denote $L=4D/d>4/d$.

As in the proof of Theorem \ref{bilip1} we define the map $f:E\to f(E)\subset F$ by 
$$f(x)=\lim_{k\to\infty}f_k(x_{i_1\dots i_k})$$
when $x=\lim_{k\to\infty}x_{i_1\dots i_k}$. Then bilip$(f)\leq L$.

If $E$ is bounded and $F$ unbounded, the same proof works with $F$ replaced by $F\cap B(p,1)$ for any $p\in F$.
Suppose $E$ and $F$ are unbounded, and let $p\in E$. Using the proof of Lemma \ref{rings} we find 
$R_k, (2D)^k\leq R_k\leq D(2D)^k, k=1,2,\dots$, such that  
$$E\cap B(p,R_k+(2D)^k)\setminus B(p,R_k)=\emptyset.$$
Let $E_k=E\cap B(p,R_k)$. We check that $E_k$ is $s$-regular with $C_{E_k}\leq(2D)^sC_E$. To see this, let 
$x\in E_k$ and $0<r\leq d(E_k)\leq(2D)^{k+1}$. If $r\leq(2D)^k$, then $E_k\cap B(x,r)=E\cap B(x,r)$, so 
$\mu(E_k\cap B(x,r))\geq r^s$. If $r>(2D)^k$, we have $\mu(E_k\cap B(x,r))\geq(2D)^{ks}\geq(2D)^{-s}r^s$. 
These facts imply that $C_{E_k}\leq(2D)^sC_E$. Since the sets $E_k$ are bounded we can find bilipschitz maps 
$f_k:E_k\to f(E_k)\subset F$ with bilip$(f_k)\leq L$ where $L$ depends only on $s,t,C_E$ and $C_F$. Using 
Arzela-Ascoli theorem we can extract a subsequence $(f_{k_i})$ such that the sequence $(f_{k_i})_{k_i\geq k}$ 
converges on $E_k$ for every $k=1,2,\dots$. Then $f=\lim_{i\to\infty}f_{k_i}:E\to f(E)\subset F$ is 
bilipschitz with bilip$(f)\leq L$.

\end{proof}

\section{mappings in $\R^n$} 
In this section we prove for small dimensional sets in $\R^n$ that we can find bilipschitz 
mappings of the whole $\R^n$. The 
following lemma may be well known, but we have not found a suitable reference in literature.

\begin{lemma}\label{lemma}
Let $0<\delta<c(n)$, where $c(n)<1/2$ is a positive constant depending only on 
$n$ and determined later. Let $p,q\in\R^n$ and $R>0$. For $i=1,\dots,m$ let $\delta R\leq r_i\leq R/3$ and 
$x_i\in B(p,R)$ and $y_i\in B(q,R)$ with 
$B(x_i,3r_i)\cap B(x_j,3r_j)=\emptyset$ and $B(y_i,3r_i)\cap B(y_j,3r_j)=\emptyset$ for $i\not=j$. 
Then there is a bilipschitz map $f:\R^n\to\R^n$ 
such that $f(x)=x-p+q$ for $x\in\R^n\setminus B(p,2R)$ and $f(x)=x-x_i+y_i$ for $x\in B(x_i,r_i)$. 
Moreover, $bilip(f)\leq L$ where $L$ depends only on $n$ and $\delta$.
\end{lemma}
\begin{proof} We may assume that $p=q=0$ and $R=1$. 
Let $\epsilon=\delta^{2n+3}$. It is enough to construct a bilipschitz map $f:\R^n\to\R^n$  with bilip$(f)\leq L, L$ 
depending only on $n$ and $\delta$, such that $f(x)=x$ for $|x|>3\sqrt{n}$ and $f(x)=x-x_i+y_i$ for 
$x\in B(x_i,\epsilon)$. To see this, consider bilipschitz maps $g,h:\R^n\to\R^n$  with bilipschitz constants 
depending only on $n$ and $\delta$ such that $g(x)=(\epsilon/r_i)(x-x_i)+x_i$ for $x\in B(x_i,r_i), g(x)=x$ 
for $x\in B(0,3/2)\setminus\bigcup_{i=1}^mB(x_i,2r_i)$, 
$h(y)=(\epsilon/r_i)(y-y_i)+y_i$ for $y\in B(y_i,r_i)$, $h(y)=y$ for 
$y\in B(0,3/2)\setminus\bigcup_{i=1}^mB(y_i,2r_i)$, $g(x)=h(x)$ for $|x|>2$ and 
$g(B(0,2))=h(B(0,2))=B(0,3\sqrt{n})$. Then $h^{-1}\circ f\circ g$ has the required properties. 

For the rest of the proof we assume that $n\geq2$, for $n=1$ a much simpler argument works. 
Denote $Q=[-2,2]^{n-1}$. Let $a,b\in B(0,1)\subset\R^{n-1}$. For $v\in\partial B(a,\epsilon)$ 
denote by $v'$ the single point 
in $\partial Q\cap\{t(v-a)+a:t\geq1\}$. Let $g(a,b):Q\to Q$ be the bilipschitz map such that 
$$g(a,b)(x)=x-a+b\ \text{for}\ x\in B(a,\epsilon)$$
and for $v\in\partial B(a,\epsilon)$ $g(a,b)$ maps the line segment 
$[v,v']$, affinely onto the line segment $[v-a+b,v']$. Then $g(a,b)(x)=x$ for $x\in\partial Q$ and $g(a,a)$ 
is the identity map. Moreover, $g(a,b)$ has a bilipschitz constant depending only on $n$.

Now we show that there exists a unit vector $\theta\in S^{n-1}$ such that $|\theta\cdot(x_i-x_j)|>5\epsilon$ 
and $|\theta\cdot(y_i-y_j)|>5\epsilon$ for $i\not=j$. To see this, let $\sigma$ denote the surface measure on 
$S^{n-1}$. We have by some simple geometry (or one can consult \cite{M}, Lemma 3.11) 
$$\sigma(\{\theta\in S^{n-1}: |\theta\cdot(x_i-x_j)|\leq5\epsilon\})\leq C_1(n)|x_i-x_j|^{-1}\epsilon
\leq C_1(n)\delta^{2n+2},$$
and similarly for $y_i, y_j$. There are less than $C_2(n)\delta^{-2n}$ pairs 
$(x_i,x_j)$ and $(y_i,y_j)$, whence 
\begin{align*}
&\sigma(\{\theta\in S^{n-1}: |\theta\cdot(x_i-x_j)|\leq5\epsilon\ \text{or}\  
|\theta\cdot(y_i-y_j)|\leq5\epsilon\ \text{for some}\ i\not=j\})\\
&<\delta,
\end{align*}
if $C_1(n)C_2(n)\delta<1$, which we have taking $c(n)\leq(C_1(n)C_2(n))^{-1}$ in the statement of the theorem. 
Taking also $c(n)\leq\sigma(S^{n-1})$ our $\theta$ exists. We may assume that $\theta=(0,\dots,0,1)$.

Let $t_i$ and $u_i, i=1,\dots,m$, be the $n$'th coordinates of $x_i$ and $y_i$, respectively, 
and let $t_0=u_0=-2, t_{m+1}=u_{m+1}=2$. 
We may assume that $t_i<t_{i+1}$ and $u_i<u_{i+1}$ for $i=0,\dots,m$. Then 
$|t_i-t_j|>5\epsilon$ and $|u_i-u_j|>5\epsilon$ for $i\not=j, i,j=0,\dots,m+1$. For $x=(x^1,\dots,x^n)\in\R^n$, 
let $\tilde x=(x^1,\dots,x^{n-1})$. Let $Q_0=[-2,2]^n$ and for $i=1,\dots,m$, 
$$R_i=\{x\in Q_0: |x^n-t_i|\leq\epsilon\},$$
$$S_i=\{y\in Q_0: |y^n-u_i|\leq\epsilon\}.$$

We shall define $f$ in $Q_0$ with the help of the maps $g(a,b)$ in such a way that it maps $R_i$ onto $S_i$ 
translating $B(x_i,\epsilon)$ onto $B(y_i,\epsilon)$. Between $R_i$ and $R_{i+1}$ $f$ is defined by simple 
homotopies changing $f|R_i$ to $f|R_{i+1}$, and similarly in $Q_0$ 'below' $R_1$ and 'above' $R_m$. Finally 
$f$ can be extended from $Q_0$ to all of $\R^n$ rather trivially. We do this now more precisely. 

Let $x\in Q_0$ and $1\leq i\leq m+1$. We set
\begin{align*}
&f(x)=(g(\tilde{x_i},\tilde{y_i})(\tilde x),x^n-t_i+u_i)\ \text{if}\ |x^n-t_i|\leq\epsilon\ \text{and}\ i\leq m,\\ 
&f(x)=(g((2\epsilon-|x^n-t_i|)/\epsilon)\tilde{x_i},(2\epsilon-|x^n-t_i|)/\epsilon)\tilde{y_i})(\tilde x),
x^n+u_i-t_i)\\ 
&\text{if}\ \epsilon\leq|x^n-t_i|\leq2\epsilon\ \text{and}\ i\leq m,\\
&f(x)=(\tilde x,\frac{x^n-t_{i-1}-2\epsilon}{t_i-t_{i-1}-4\epsilon}(u_i-2\epsilon)+
\frac{t_i-2\epsilon-x^n}{t_i-t_{i-1}-4\epsilon}(u_{i-1}+2\epsilon))\\ 
&\text{if}\ t_{i-1}+2\epsilon\leq x^n\leq t_i-2\epsilon,\\
&f(x)=x\ \text{if}\ -2\leq x^n\leq-2+2\epsilon\ \text{or}\ 2-2\epsilon\leq x^n\leq2.
\end{align*}
Then $f:Q_0\to Q_0$ is bilipschitz with a constant depending only on $n$ and $\delta$, 
$f(x)=x-x_i+y_i$ for $x\in B(x_i,\epsilon)$, $f(x)=x$ for $x\in Q_0$ with 
$x_n=-2$ or $x_n=2$, and at the other parts of the boundary of $Q_0$ $f$ is of the form $f(x)=(\tilde x, \phi(x^n))$ 
where $\phi:[-2,2]\to[-2,2]$ is strictly increasing and piecewise affine. It is an easy matter to extend $f$ 
to a bilipschitz mapping 
of $\R^n$ with a bilipschitz constant depending only on $n$ and $\delta$ and with $f(x)=x$ for 
$x\in\R^n\setminus B(0,3\sqrt{n})$. For example, setting $||\tilde x||_{\infty}=\max\{|x^1|,\dots,|x^{n-1}|\}$, 
we can take 
$$f(x)=(\tilde x, (3-||x||_{\infty})\phi(x^n)+(||x||_{\infty}-2)x^n)$$
when $2\leq||\tilde x||_{\infty}\leq3$ and $|x^n|\leq2$, and $f(x)=x$ when $||\tilde x||_{\infty}>3$ or $|x^n|>2$.
\end{proof}

\begin{thm}\label{bilip3}
Let $C\geq1$ and let $s_0=s_0(C,18), 0<s_0<1/6$, be the constant of Lemma \ref{rings2}. Let 
$0<s<s_0$ and $s<t<n$, let $E\subset\R^n$ be $s$-regular and $F\subset\R^n$ $t$-regular 
with $C_E, C_F\leq C$. 
Suppose that either $E$ is bounded or both $E$ and $F$ are unbounded. 
Then there is a bilipschitz map $f:\R^n\to\R^n$ 
such that $f(E)\subset F$.
\end{thm}
\begin{proof} We assume that $E$ and $F$ are bounded. The remaining case can be dealt with as at the end 
of the proof of Theorem \ref{bilip2}. We can then assume that $E,F\subset B(0,1)$ 
with $d(E)=d(F)=1/2$. Let $c(n)$ and $D=D(C,18)$ 
be as in Lemma \ref{rings2}, and choose $d$ such that
$$d<c(n), 12Dd<1\ \text{and}\ 0<d^{t-s}<(2^s60^tC_EC_FD^t)^{-1}.$$
By Lemma \ref{rings2} we find $x_i\in E$ and $\rho_i, d\leq \rho_i\leq Dd$, with 
$i=1,\dots,m_0, m_0\leq C_E/d^s$,
such that the balls $B(x_i,6\rho_i)$ are disjoint, 
$$E\subset\bigcup_{i=1}^{m_0}B(x_i,\rho_i)$$
and
$$E\cap B(x_i,18\rho_i)\setminus B(x_i,\rho_i)=\emptyset.$$
By Lemma \ref{balls} we find  $y_i\in F$ with 
$i=1,\dots,n_0$,\\ 
$n_0\geq (5^tC_F)^{-1}(1/(12Dd))^t\geq C_E/d^s\geq m_0$
such that the balls $B(y_i,6Dd)$, $j=1,\dots,n_0$,  are disjoint. 
Next applying Lemma \ref{rings2} with $E$ replaced by $E\cap B(x_i,\rho_i), R=d, r=d^2$ and $C=C_E$, 
we find for every $i=1,\dots,m_0, x_{ij}\in E\cap B(x_i,\rho_i)$
and $\rho_{ij}, d^2\leq\rho_{ij}\leq Dd^2$, with $j=1,\dots,m_i, 
m_i\leq C_Ed(E\cap B(x_i,\rho_i))^s/d^{2s}\leq C_E2^sD^s/d^s$, 
such that the balls $B(x_{ij},6\rho_{ij})$ are disjoint, 
$$E\cap B(x_i,\rho_i)\subset\bigcup_{j=1}^{m_i} B(x_{ij},\rho_{ij})$$
and
$$E\cap B(x_{ij},18\rho_{ij})\setminus B(x_{ij},\rho_{ij})=\emptyset,$$
and by Lemma \ref{balls} we find  $y_{ij}\in F\cap B(y_i,d), j=1,\dots,n_i$, 
$n_i\geq (5^tC_F)^{-1}(d/(6Dd^2))^t\geq 2^sC_E/d^s\geq m_i$
such that the balls $B(y_{ij},6Dd^2)$ are disjoint. 
Continuing this we find for all $k=1,2,\dots$, $x_{i_1\dots i_k}, \rho_{i_1\dots i_k}$ and 
$y_{i_1\dots i_k}$ such that for all $i_j=m_{i_0\dots i_{j-1}}, j=1,\dots,k, k=1,2,\dots,$ with $i_0=0$,
\begin{align*}
&E\subset\bigcup_{i_1\dots i_k}B(x_{i_1\dots i_k},\rho_{i_1,\dots,i_k}),\\
&B(x_{i_1\dots i_k},6\rho_{i_1\dots i_k})\cap B(x_{j_1\dots j_k},6\rho_{i_1\dots i_k})
=\emptyset\ \text{if}\ i_k\not=j_k,\\
&x_{i_1\dots i_ki_{k+1}}\in E\cap B(x_{i_1\dots i_k},\rho_{i_1\dots i_k}),\\
&d^k\leq\rho_{i_1\dots i_k}\leq Dd^k,\\
&B(x_{i_1\dots i_ki_{k+1}},4\rho_{i_1\dots i_ki_{k+1}})\subset B(x_{i_1\dots i_k},2\rho_{i_1\dots i_k}),\\
&E\cap B(x_{i_1\dots i_k},18\rho_{i_1\dots i_k})\setminus B(x_{i_1\dots i_k},\rho_{i_1\dots i_k})=\emptyset,\\
&y_{i_1\dots i_ki_{k+1}}\in F\cap B(y_{i_1\dots i_k},d^k),\\
&B(y_{i_1\dots i_k},6Dd^k)\cap B(y_{j_1\dots j_k},6Dd^k)=\emptyset\ \text{if}\ i_k\not=j_k.
\end{align*}

Using Lemma \ref{lemma} we find a bilipschitz map $f_1:\R^n\to\R^n$  
such that $f_1(x)=x$ for $|x|>2$ and $f_1(x)=x-x_i+y_i$ for $x\in B(x_i,2\rho_i)$, and 
bilip$(f)\leq L$ where $L$ depends only on $s, t, n$ and $C$.
Let
$$B_k=\bigcup_{i_1\dots i_k}B(x_{i_1\dots i_k},2\rho_{i_1,\dots,i_k}).$$
Then $B_{k+1}\subset B_k$ for all $k$ and $E=\bigcap_{k=1}^{\infty}B_k$.
We use Lemma \ref{lemma} to define inductively $f_k:\R^n\to\R^n$  
such that $f_{k+1}(x)=f_k(x)$ for $x\in\R^n\setminus B_k^o$, where $B_k^o$ is the interior of $B_k$, 
$f_{k+1}|B(x_{i_1\dots i_k},2\rho_{i_1,\dots,i_k})$ is $L$-bilipschitz 
and $f_{k+1}(x)=x-x_{i_1,\dots,i_{k+1}}+y_{i_1,\dots,i_{k+1}}$ for 
$x\in B(x_{i_1\dots i_{k+1}},2\rho_{i_1,\dots,i_{k+1}})$.
We check now by induction that
\begin{equation}\label{bili}
|x-y|/L\leq|f_k(x)-f_k(y)|\leq L|x-y|\ \text{for all}\ x,y\in\R^n.
\end{equation}
For $k=1$ this was already stated. Suppose this is true for $k-1$ for some $k\geq2$ and let $x,y\in\R^n$.
If  $x,y\in\R^n\setminus B_k^o$, (\ref{bili}) follows from the definition of $f_k$ and the induction hypothesis.
If $x,y\in B(x_{i_1\dots i_k},2\rho_{i_1,\dots,i_k})$ for some $i_1\dots i_k$, then (\ref{bili}) follows from 
the fact that $f_k$ is a translation in $B(x_{i_1\dots i_k},2\rho_{i_1,\dots,i_k})$. 
Finally, let $x\in B(x_{i_1\dots i_k},2\rho_{i_1,\dots,i_k})$ and 
$y\in\R^n\setminus B(x_{i_1\dots i_k},2\rho_{i_1,\dots,i_k})$.
Let $z\in\partial B(x_{i_1\dots i_k},2\rho_{i_1,\dots,i_k})$ be the point on the line segment with end points 
$x$ and $y$. Then, by the two previous cases,  
\begin{align*}
&|f_k(x)-f_k(y)|\leq|f_k(x)-f_k(z)|+|f_k(z)-f_k(y)|\leq\\
& L|x-z|+L|z-y|=L|x-z|.
\end{align*}
This proves the right hand inequality of (\ref{bili}). A similar argument for $f_k^{-1}$ with the balls 
$B(y_{i_1\dots i_k},2\rho_{i_1,\dots,i_k})$ gives the left hand inequality.

We have left to show that the limit $\lim_{k\to\infty}f_k(x)=f(x)$ exists for all $x\in\R^n$. Then also $f$ satisfies 
(\ref{bili}) and $f(E)\subset F$. First, if $x\in\R^n\setminus E$, then  $x\in\R^n\setminus B_l$ for some 
$l$, and so $f_k(x)=f_l(x)$ for $k\geq l$. If $x\in E$, there are $i_1,i_2,\dots,$ such that 
$x\in B(x_{i,\dots i_k},2\rho_{i_1\dots i_k})$ for all $k$. Then $f_k(x)\in B(y_{i_1\dots i_k}, 2Dd^k)$ and 
$\lim_{k\to\infty}f_k(x)=y=f(x)$ where $y=\lim_{k\to\infty}y_{i_1\dots i_k}$.

\end{proof}

\section{sub- and supersets}
In this section we shall consider the question whether a given regular set contains regular subsets 
of smaller dimension and whether it is contained in higher dimensional regular sets.

\begin{thm}\label{subset} Let $E\subset X$ be $s$-regular and $0<t<s$. For every $x\in E$ and $0<r<d(E)$, 
$E\cap B(x,r)$ contains a $t$-regular subset $F$ such that $C_F\leq C$ and $d(F)\geq cr$ where $C$ and $c$ are 
positive constants depending only on $s,t$ and $C_E$.
\end{thm}

This can be proven with the same method as Theorem \ref{bilip1}. In fact, that method gives that $E\cap B(x,r)$ 
has a $t$-regular subset which is bilipschitz equivalent with $C(t,r)$ with a bilipschitz constant 
depending only on $s,t$ and $C_E$. Observe that the regularity of $E$ implies 
that $d(E\cap B(x,r))\geq C_E^{-1/s}r$. 

\begin{thm}\label{superset} Let $0<s<t<u$. Suppose that $E\subset X$ is $s$-regular and that $X$ is $u$-regular.
Then there is a $t$-regular set $F$ with $E\subset\ F\subset X$. Moreover, $C_F\leq C$ where $C$ depends only 
on $s, t, C_E$ and $C_X$.
\end{thm}

\begin{proof}  
We shall only consider the case where $X$ and $E$ are bounded. A slight modification 
of the proof works if $X$ or both $X$ and $E$ are unbounded. 
Recalling the remarks at the beginning of Section 3, we may assume that $d(E)=1$.
Let $0<d<1/30$ be such that $d^{u-s}<4^{-s}30^{-u}C_E^{-1}C_X^{-1}$. 
By Lemma \ref{balls} there are for every $k=1,2\dots,$ disjoint balls 
$B(x_{k,i},6d^k), i=1,\dots,m_k$, such that $x_{k,i}\in E$ and the 
balls $B(x_{k,i},30d^k)$ cover $E$. Further, there are disjoint balls 
$B(x_{k,i},6d^k), i=m_k+1,\dots,n_k$, such that 
$x_{k,i}\in X\setminus\cup_{i=1}^{m_k}B(x_{k,i},30d^k)$ and the 
balls $B(x_{k,i},30d^k), i=1,\dots,n_k$, cover $X$.

Fix $k$ and $i,1\leq i\leq m_k$. Denote
\begin{align*}
&J=\{j\in\{1,\dots,n_k\}: B(x_{k+1,j},d^{k+1})\subset B(x_{k,i},3d^k)\},\\
&J'=\{j\in\{1,\dots,n_k\}: B(x_{k+1,j},30d^{k+1})\cap B(x_{k,i},d^k)\not=\emptyset\},\\
&I=\{j\in J:E\cap B(x_{k+1,j},6d^{k+1})\not=\emptyset\},\\
\end{align*}
and let $n,n'$ and $m$ be the number of indices in $J,J'$ and $I$, respectively. 
Then, as $d<2/31$,  $J'\subset J$ and so $n'\leq n$.
Since $B(x_{k,i},d^k)\subset \cup_{j\in J'}B(x_{k+1,j},30d^{k+1})$, we have, comparing 
measures as in the proof of Lemma \ref{balls}, that $n\geq n'\geq30^{-u}C_X^{-1}d^{-u}$. 
If $j\in I$ there is $z_j\in E\cap B(x_{k+1,j},6d^{k+1})$ and then, as also $j\in J$ and $d<1/7$, 
$$B(z_j,d^{k+1})\subset B(x_{k+1,j},7d^{k+1})\subset B(x_{k,i},4d^k).$$
Then the balls $B(z_j,d^{k+1}), j\in I,$ are disjoint and 
$$md^{(k+1)s}\leq\sum_{j\in I}\mu(B(z_j,d^{k+1}))\leq\mu(B(x_{k,i},4d^k))\leq4^sC_Ed^{ks},$$
whence $m\leq 4^sC_Ed^{-s}$. Combining these inequalities and recalling the choice of 
$d$, we find that
$$m\leq 4^sC_Ed^{-s}<30^{-u}C_X^{-1}d^{-u}\leq n.$$
Thus we can choose some $j\in J\setminus I$. Let $y_{k,i}=x_{k+1,j}$ and $B_{k,i}=B(y_{k,i},d^{k+1})$.
Denote also $2B_{k,i}=B(y_{k,i},2d^{k+1})$. Then for a fixed $k$ the balls $2B_{k,i},i=1,\dots,m_k$, 
are disjoint. If $x\in2B_{k,i}$, then, as $j\not\in I$, $d(x,E)\geq4d^{k+1}$. On the other hand, 
as $j\in J$, $d(x,E)\leq d(x,x_{k,i})\leq d(x,y_{k,i})+d(y_{k,i},x_{k,i})\leq2d^{k+1}+3d^k<d^{k-1}$. 
It follows that the balls 
$2B_{k,i}$ and $2B_{l,j}$ with $|k-l|\geq2$ are always disjoint. Hence any point of $X$ can belong to 
at most two balls $2B_{k,i}, i=1,\dots,m_k, k=1,2,\dots$. 

By Theorem \ref{subset} we can choose for every $k,i,1\leq i\leq m_k$, $t$-regular sets 
$F_{k,i}\subset B_{k,i}$ such that $C_{F_{k,i}}\leq C$ and $d(F_{k,i})\geq cd^k$ with $C$ and $c$ 
depending only on $t,u$ and $C_X$. Let $\nu_{k,i}$ be 
the Borel measure related to $F_{k,i}$ as in Definition \ref{regular}. We define
$$F=E\cup\bigcup_{k=1}^{\infty}\bigcup_{i=1}^{m_k}F_{k,i}$$
and
$$\nu=\sum_{k=1}^{\infty}\sum_{i=1}^{m_k}\nu_{k,i}.$$
Then $F$ is a closed and bounded subset of $X$ containing $E$. 

We check now that $F$ is $t$-regular. Let $x\in F$ and $0<r\leq d(F)$. It is enough to verify the required inequalities for 
$r<d$, so we assume this. Let $l$ be the positive integer for which $d^{l+1}\leq r<d^l$. Denote
$$K=\{(k,i):i=1,\dots,m_k, k<l\ \text{and}\ B_{k,i}\cap B(x,r)\not=\emptyset\}$$
and
$$L=\{(k,i):i=1,\dots,m_k, k\geq l\ \text{and}\ B_{k,i}\cap B(x,r)\not=\emptyset\}.$$
We have 
$$\nu(B(x,r))\leq\sum_{(k,i)\in K}\nu_{k,i}(B_{k,i}\cap B(x,r))+\sum_{(k,i)\in L}\nu_{k,i}(B_{k,i}\cap B(x,r)).$$
If $(k,i)\in K$, then $r<d^{k+1}$ and $B(x,r)\subset2B_{k,i}$. Since this can happen for at most two balls 
$2B_{k,i}$, $K$ can contain at most two 
elements and the first sum above is bounded by $2^{t+1}Cr^t$. To estimate the second sum, let 
$p_k$ be the number of indices in $I_k=\{i:(k,i)\in L\}$. Let $(k,i)\in L$.  
Then $B_{k,i}\cap B(x,r)\not=\emptyset$, and so $d(x_{k,i},x)\leq d(x_{k,i},y_{k,i})+d(y_{k,i},x)
\leq3d^k+2d^{k+1}+r<5d^l$, which gives
$B(x_{k,i},d^k)\subset B(x,6d^l)$. Consequently,
$$p_kd^{ks}\leq\sum_{i\in I_k}\mu(B(x_{k,i},d^k))\leq\mu(B(x,6d^l))\leq C_E12^sd^{ls},$$
and so $p_k\leq12^sC_Ed^{(l-k)s}$.
Hence 
\begin{align*}
&\sum_{(k,i)\in L}\nu_{k,i}(B_{k,i}\cap B(x,r))\leq\sum_{k=l}^{\infty}12^sC_Ed^{(l-k)s}C4^td^{(k+1)t}\leq\\
&12^s4^tC_ECd^{ls}\sum_{k=l}^{\infty}d^{(t-s)k}=12^s4^tC_ECd^{lt}\frac{1}{1-d^{(t-s)}}\leq\\
&12^s4^tC_ECd^{-t}\frac{1}{1-d^{(t-s)}}r^t.
\end{align*}
This proves the upper regularity of $\nu$.

To prove the opposite inequality, suppose first that $x\in E$. Let $k$ be the positive integer 
for which $33d^k\leq r<33d^{k-1}$. 
Then for some $i, 1\leq i\leq m_k, x\in B(x_{k,i},30d^k)$. Since 
$B_{k,i}\subset B(x_{k,i},3d^k)$ we have that $B_{k,i}\subset B(x,33d^k)\subset B(x,r)$. Thus
$$\nu(B(x,r))\geq \nu_{k,i}(B_{k,i})\geq d(F_{k,i})^t\geq c^td^{kt}\geq c^td^t33^{-t}r^t.$$
Suppose finally that $x\in F_{k,i}$ for some $k$ and $i$. If $r\leq9d^k$, then $d(F_{k,i})\geq cd^k\geq(c/9)r$,
whence 
$$\nu(B(x,r))\geq \nu(B(x,(c/9)r)\geq(c/9)^tr^t.$$
If $r>9d^k$, then 
$d(x,x_{k,i})\leq 3d^k<r/3$, 
so $B(x_{k,i},r/2)\subset B(x,r)$. Since $x_{k,i}\in E$, the 
required inequality follows from the case $x\in E$.

\end{proof}

In the next example note that $\lim_{r\to0}\mathcal L^1(F\cap B(x,r))/(2r)=1$ for $\mathcal L^1$ almost all 
$x\in F$ by the Lebesgue density theorem. However, $F$ has no subset $E$ with $\mathcal L^1(E)>0$ for which 
$\mathcal L^1(F\cap B(x,r))/(2r)$ would be bounded below with a positive number uniformly for small $r>0$.

\begin{exmp}\label{example}
There exists a compact set $F\subset\R$ with Lebesgue measure $\mathcal L^1(F)>0$ such that 
it contains no non-empty $s$-regular subset for any $s>0$.
\end{exmp}
\begin{proof} Let $a<b, 0<\lambda<1/2$ and $0<t<1$. We shall construct a family 
$\mathcal I([a,b],\lambda,t)$ 
of closed disjoint subintervals of $[a,b]$. We do this for  $[0,1]$ and then define 
$\mathcal I([a,b],\lambda,t)=\{f(I):I\in\mathcal I([0,1],\lambda,t)\}$ where $f(x)=(b-a)x+a$.

Let
$$I_{1,1}=[(1-\lambda)/2,(1+\lambda)/2].$$
Then $[0,1]\setminus I_{1,1}$ consists of two intervals $J_{1,1}$ and $J_{1,2}$ of length $(1-\lambda)/2$.  
We select closed intervals $I_{2,1}$ and $I_{2,2}$ of length $\lambda(1-\lambda)/2$ in the middle of them 
(that is, the center of $I_{2,i}$ is the center of $J_{1,i})$. 
Continuing this we get intervals $I_{k,i}, i=1,\dots,2^{k-1},$ and $J_{k,i}, i=1,\dots,2^k,$ such that 
$d(I_{k,i})=2^{1-k}\lambda(1-\lambda)^{k-1}$ and $d(J_{k,i})=2^{-k}(1-\lambda)^k$. 
Moreover, each $I_{k,i}$ is the mid-interval 
of some $J_{k-1,j}$ and $J_{k-1,j}\setminus I_{k,i}$ consists of two intervals $J_{k,j_1}$ and $J_{k,j_2}$. Then
\begin{align*}
&\sum_{k=1}^l\sum_{i=1}^{2^{k-1}}d(I_{k,i})=\sum_{k=1}^l\lambda(1-\lambda)^{k-1}\\
&=1-(1-\lambda)^l\to 1\ \text{as}\ l\to\infty.
\end{align*}
We choose $l$ such that
$$\sum_{k=1}^l\sum_{i=1}^{2^{k-1}}d(I_{k,i})>t$$
and denote
$$\mathcal I([0,1],\lambda,t)=\{I_{k,i}:i=1,\dots,2^{k-1}, k=1,\dots,l\}.$$
Then for any compact interval $I\subset\R$,
$$\sum_{J\in\mathcal I(I,\lambda,t)}d(J)>td(I).$$

Let $0<\lambda_k<1/2, 0<t_k<1, k=1,2,\dots$, such that $\lim_{k\to\infty}\lambda_k=0$ and $t=\prod_{k=1}^{\infty}t_k>0$.
Define
$$\mathcal I_1=\mathcal I([0,1],\lambda_1,t_1),$$
and inductively for $m=1,2,\dots,$
$$\mathcal I_{m+1}=\{J:J\in\mathcal I(I,\lambda_{m+1},t_{m+1}), I\in\mathcal I_m\}.$$
The compact set $F$ is now defined as
$$F=\bigcap_{m=1}^{\infty}\bigcup_{I\in \mathcal I_m}I.$$
For every $m=1,2,\dots$ we have
$$\sum_{I\in\mathcal I_m}d(I)>t_1\cdot\dots\cdot t_m>t,$$ 
whence $\mathcal L^1(F)\geq t$.

Suppose that $s>0$ and that $E$ is an $s$-regular subset of $F$. 
Choose $m$ so large that $\lambda_m<C_E^{-s}/4$. Let $x\in E$. Then for every $m=1,2,\dots, x\in I$ 
for some $I\in\mathcal I_m$. Suppose that $I$ would be one of the shortest intervals in the family $\mathcal I_m$. 
Then by our construction there is an 
interval $J$ such that $I$ is in the middle of $J, I\cap E=J\cap E$ and $d(I)=\lambda_md(J)$. As $B(x,d(J)/4)\subset J$ we 
have by the regularity of $E$,
\begin{align*}
&4^{-s}d(J)^s\leq \mu(B(x,d(J)/4))=\\
&\mu(B(x,d(I))\leq C_Ed(I)^s=C_E(\lambda_md(J))^s.
\end{align*}
Thus $\lambda_m\geq C_E^{-1/s}/4$. This contradicts with the choice of $m$. So $E$ contains no points in the 
shortest intervals of $\mathcal I_m$. But then we can repeat the same argument with the second shortest 
intervals of $\mathcal I_m$ concluding that neither can they contain any points of $E$. Continuing this we 
see that $E=\emptyset$.

\end{proof}

\end{document}